%% file: Kirk.tex
\documentclass[11pt]{amsart}
\usepackage{amssymb, amscd, amsmath, amsthm, epsf, epsfig, latexsym,color} 
\usepackage{hyperref}

\usepackage{tikz}
 \usepackage{tikz-cd}

\usepackage{comment}
\usepackage{cancel}   
\usepackage{mathtools}
\usepackage{pb-diagram}

\newcommand{\ta}{\tilde{a}}

\newcommand{\tp}{\tilde{p}}

\newcommand{\TryPackage}[3]{\IfFileExists{#1.sty}{\usepackage{#1}#2}{#3}
}
\TryPackage{mathrsfs}{\renewcommand{\mathcal}{\mathscr}}{%
        \TryPackage{eucal}{}{}}

\newcommand{\ep}{\epsilon}

\newcommand{\bbi}{{{\bf i}}}
\newcommand{\bbj}{{{\bf j}}}
\newcommand{\bbk}{{{\bf k}}}

\newcommand{\HH}{{\mathbb H}}
\newcommand{\FF}{{\mathbb F}}
\newcommand{\ZZ}{{\mathbb Z}}
\newcommand{\RR}{{\mathbb R}}
\newcommand{\CC}{{\mathbb C}}

\newcommand{\TT}{{\mathbb T}}

\newcommand{\RN}[1]{%
  \textup{\uppercase\expandafter{\romannumeral#1}}%
}

\newcommand{\Hom}{\operatorname{Hom}}

\newcommand{\Real}{\operatorname{Re}}

\graphicspath{ {figures/} }

\theoremstyle{definition}

\newtheorem{df}{Definition}[section]

\theoremstyle{plain}

\newtheorem{thm}[df]{Theorem}

\newtheorem{cor}[df]{Corollary}
\newtheorem{lem}[df]{Lemma}
\newtheorem{prop}[df]{Proposition}

\date{}

\thanks{ }

 \author{Paul Kirk}
\address{Department of Mathematics, Indiana University, Bloomington, IN 47405} 
\email{pkirk@indiana.edu}


\begin{document}

\title{On the traceless SU(2) character variety of the 6-punctured 2-sphere}

\thanks{The author thanks the Simons Foundation for its support through the collaboration grant 278714 and the Max-Planck-Institut f\"ur Mathematik for its support during the Fall 2015 semester.}

\begin{abstract} We exhibit the traceless $SU(2)$ character variety of a 6-punctured 2-sphere as a 2-fold branched cover of $\CC P^3$, branched over the singular Kummer surface, with the branch locus in $R(S^2,6)$ corresponding to the binary dihedral representations. This follows from an analysis of the map induced on $SU(2)$ character varieties by the 2-fold branched cover $F_{n-1}\to S^2$ branched over $2n$ points, combined with the theorem of Narasimhan-Ramanan which identifies $R(F_2)$ with $\CC P^3$. The singular points of $R(S^2,6)$ correspond to abelian representations, and we prove that each has a neighborhood in $R(S^2,6)$ homeomorphic to a cone on $S^2\times S^3$.

\end{abstract}

 \maketitle
 
 \section{Introduction}
 The traceless character variety of a $k$-punctured 2-sphere is the space $R(S^2,k)$ of conjugacy classes of $SU(2)$ representations of the fundamental group of $S^2\setminus \{a_i\}_{i=1}^k$ which send each loop encircling a puncture to a traceless matrix in $SU(2)$.   It contains a binary dihedral locus $R(S^2, k)_{B}$ and an abelian locus $R(S^2,k)_{\rm ab}$ (see Definitions \ref{bdloc} and \ref{abloc}) satisfying $R(S^2,k)_{\rm ab}\subset R(S^2,k)_{B}\subset R(S^2,k)$. 
 
 Similarly, for a closed orientable surface, let $R(F)$ denote the space of conjugacy classes of $SU(2)$ representations of the fundamental group of $F$. It contains an abelian locus $R(F)_{\rm ab}$ and a central locus $R(F)_{\rm cen}$, with $R(F)_{\rm cen}\subset R(F)_{\rm ab}\subset R(F)$.
 
Let $F_2$ denote the closed orientable genus 2-surface and $p:F_2\to S^2$ the 2-fold branched cover, branched over six points.  A theorem of Ramanan-Seshadri identifies $R(F_2)$ with $\CC P^3$ and its abelian locus $R(F_2)_{\rm ab}$ with the singular Kummer surface $K\cong \TT^4/(\ZZ/2)$. Note that $K$ is a singular complex surface with $16$ nodal singularities.
 
We summarize the  main results of this article in Theorem 1, which is a amalgam of the more precise statements of Theorems \ref{main}, \ref{link2},and \ref{thm4}.
 \medskip

 \noindent{\bf Theorem 1.} {\em 
 \begin{enumerate}
\item The  2-fold branched cover $p:F_2\to S^2$  branched over six points  induces a 2-fold branched cover $\tp:R(S^2,6)\to R(F_2)$, branched over $K\cong R(F_2)_{\rm ab}$ in $\CC P^3\cong R(F_2)$,  with branch set in $R(S^2,6)$ equal to the binary dihedral locus $R(S^2,6)_B\subset R(S^2,6)$. 
\item The set of 16 singular points of $R(S^2,6)$ is equal to the abelian locus $R(S^2,6)_{\rm ab}$, and its  image in $R(F_2)$ is the central locus $R(F_2)_{\rm cen}$.  With the identification $\CC P^3\cong R(F_2)$, the singular points in $R(S^2,6)$  correspond precisely to the preimages of the 16 nodal singularities of $K\subset \CC P^3$.  
\item A neighborhood of each of these 16 points is homeomorphic to a cone on $S^2\times S^3$. \end{enumerate}}

 \medskip

We also point the reader to Corollary \ref{extend}, which gives an explicit construction of  the representations in the fiber over a point of the 2-fold branched cover $\tp:R(S^2,6)\to R(F_2)$.

 \medskip

 The proof of Theorem 1 relies on  the details of the proofs of few results about the spaces $R(S^2,k)$, some of which are known,  some of which are folklore, and some of which appear to be new. We give careful but elementary proofs of the results in the following theorem, which summarizes the statements of Propositions \ref{prop1.1}, \ref{linkages}, and \ref{prop2.6}, and Theorems \ref{thm1.2} and \ref{prop2.2}. 
 
 \medskip

  \noindent{\bf Theorem 2.} {\em
 \begin{enumerate}
 \item If $k$ is odd, $R(S^2,k)$ is smooth of dimension $2k-6$  and admits a Morse function  with only even index critical points. 
 \item   $R(S^2,2n)\setminus R(S^2,2n)_{\rm ab}$  is smooth  of dimension $4n-6$.
 \item $R(S^2, 2n)_B$ is homeomorphic to $\TT^{2n-2}/(\ZZ/2)$ for the involution on the $2n-2$-torus which takes each coordinate to its complex conjugate.
 \item $R(S^2,2n)_{\rm ab}$ is a finite set with $2^{2n-2}$ points. Each $\rho\in R(S^2,2n)_{\rm ab}$ has a neighborhood in $R(S^2,2n)$ homeomorphic to a cone on $(S^{2n-3}\times S^{2n-3})/S^1$ for some free $S^1$ action, and a neighborhood in $R(S^2, 2n)_{B}$ homeomorphic to $\RR P^{2n-3}$.
 \item Multiplying by a central character $c:\pi_1(S^2\setminus\{a_i\}_{i=1}^{2n}\to \{\pm 1\}$ induces a $\ZZ/2$ action on $R(S^2, 2n)$ with fixed set $R(S^2,2n)_B$. The orbit space of this involution embeds in $R(F_{n-1})$.  The embedding is not surjective if $n>3$.
 \end{enumerate}}
 
 \medskip
 The spaces $R(S^2,k)$ for $k<6$ are identified in Section \ref{properties}.

 \medskip
 
 The results of this article are motivated by the project of constructing a Lagrangian-Floer theory for tangle decomposition of knots (see \cite{HHK1,HHK2}) which is expected to form a computationally simpler counterpart to the singular instanton Floer homology of knots constructed by Kronheimer-Mrowka \cite{KM1,KM2}.   
 In particular, the proposed Lagrangian-Floer theory takes place in the symplectic manifolds $R(S^2,2n)\setminus R(S^2,2n)_{\rm ab}$, and  the present article provides details about the structure of the ends of this manifold. 
 We refer the interested reader to the articles \cite{HHK1,HHK2,HK,FKP} for   more details of this relationship.

 \medskip

 Thanks to Dan Ramras,  Chris Herald, Matthias Kreck, and Michael Heusener for very helpful discussions. In particular, Ramras brought the reference  \cite{nara} to the author's attention. The author also wishes to acknowledge the influence of the results in the unpublished article \cite{JR} of Jacobsson and Rubinsztein on this article, in particular on his finding the correct statement of part iv of Theorem 2. 
 
\section{Properties of the traceless character varieties of punctured spheres}\label{properties}

Let $\HH$ denote the quaternions, and identify $SU(2)$ with the group of unit quaternions. The Lie algebra $su(2)$ is spanned by $\{\bbi,\bbj,\bbk\}$.  Let $\Real:\HH\to \RR$ be the function which returns the real part of a quaternion; this corresponds to half the trace.  Let $S^2_{\bbi}$ denote the conjugacy class of $\bbi$, so that 
$$S^2_{\bbi}=\Real^{-1}(0)=su(2)\cap SU(2)=\{a\bbi +b\bbj+c\bbk~|~ a^2+b^2+c^2=1\}=\{q\in SU(2)~|~ q^2=-1\}.$$

Each unit quaternion can be written in the form $e^{\alpha P}=\cos\alpha+\sin \alpha P$ for $P\in S^2_{\bbi}$.  This description is unique  for unit quaternions different than $\pm 1$   if one chooses $0<\alpha < \pi$.  The maximal abelian groups  in  $SU(2)$ are the circles $\{e^{\alpha P}\}$ for $P\in S^2_{\bbi}$. Two such subgroups either coincide or intersect in the center $\{\pm1\}$. In particular, if $p,q\in SU(2)$ commute and $p=e^{\alpha P}$ for some $\alpha$ satisfying $\sin\alpha\ne 0$, then $q=e^{\beta P}$ for some $\beta$.

 The function $\Real:SU(2)\to \RR$ on  unit quaternions 
corresponds to one half the trace on $SU(2)$ matrices. Hence $\Real(pq)=\Real(qp)$ for $p,q\in SU(2)$. The   preimages $\Real^{-1}(x),~ x\in [-1,1]$ are precisely the conjugacy classes in $SU(2)$.
Every circle subgroup of $SU(2)$ intersects $S^2_\bbi$ in two points, $\pm Q=S^2_\bbi\cap \{e^{\beta Q}\}.$ The conjugation action of $e^{\beta Q}$ on $S^2_\bbi$ is rotation of angle $2\beta$ with fixed points $\pm Q$.

The   bilinear form $su(2)\times su(2)\to \RR$ given by $-\Real(vw)$ is positive definite, and invariant under the conjugation action of $SU(2)$. The bilinear form $-\Real(vw)$ is not positive definite on all of $\HH$, but when $v\in su(2)$ and $w\in \HH$, $-\Real(vw)=0$ if and only if $v$ and $w$ are perpendicular.

\bigskip

\begin{df} Given a compact manifold $M$, the {\em $SU(2)$ character variety of $M$} is defined to be the topological space
$$R(M)=\Hom(\pi_1(M),SU(2))/_{\rm conjugation}$$

Given an   manifold $M$ containing a codimension two submanifold $L\subset M$, 
  call a loop $\mu\in \pi_1(M\setminus L)$ a {\em meridian} provided $\mu$ is homologous in $M\setminus L$ to one of the boundary curves of any small normal disk to $L$. Then  {\em the traceless $SU(2)$ character variety of $(M,L)$} is defined to be the space
  $$R(M,L)=\{\rho\in \Hom(\pi_1(M\setminus L),SU(2) )~|~ 
\Real(\rho(\mu))=0\text{ for all meridians } \mu \}/_{\rm conjugation}.$$
\end{df}

 Although they are traditionally called character varieties, the spaces $R(M)$ and $R(M,L)$ are in general real semi-algebraic affine sets (see \cite{LM,BCR, CulSha}).  We will use the following naive notion: call  a point $p$ in an affine real semi-algebraic set $V\subset \RR^N$ a {\em smooth point} if $p$ has a neighborhood $U$    in $\RR^N$  so that the pair $(U,U\cap V)$ is diffeomorphic to $(\RR^n,\RR^k)$ for some $k$.  Otherwise call $p$ a {\em singular point}. 

\medskip 
  
When $M$ and $L$ are oriented, one can define character varieties for codimension two pairs $(M,L)$ for other traces.  More precisely, fix a  conjugacy class in $SU(2)$  for each component of $L$ and restrict to homomorphisms  which take each oriented meridian to its corresponding conjugacy class.   These more general character varieties, which  make a brief appearance in the present article in Definition \ref{othertrace} and 
Proposition \ref{prop2.1} below,   are well studied objects when $(M,L)$ is a surface and $L$ set of points, and typically for {\em generic} choices of traces the traceless character varieties are smooth and stable under small changes of the choice of conjugacy classes.  The focus in this article on the traceless case is motivated by the relationship (explained in \cite{HHK1,HHK2}) to Kronheimer-Mrowka's {\em singular instanton homology} \cite{KM1,KM2}, which they prove is only well-defined in the traceless context.

 \medskip
 
 Let $\{a_i\}_{i=1}^k\subset S^2$ be a collection of $k$ distinct points in $S^2$, and  give $\pi_1(S^2\setminus \{a_i\}_{i=1}^k)$ the presentation
 \begin{equation}\label{presentation}
\pi_1(S^2\setminus \{a_i\}_{i=1}^k ) =\langle x_1,\dots,x_{k}~|~x_1x_2\cdots x_{k}=1\rangle,
\end{equation}
where  $x_i$ denotes the meridian which goes around the puncture $a_i$.
 
  To streamline notation, we will    denote the traceless $SU(2)$ character variety of $(S^2, \{a_i\}_{i=1}^k)$ simply by 
 $R(S^2, k)$. 
 
 \medskip

For small $k$ the following examples are known. When $k=0$, $R(S^2, 0)$ contains a single point, namely the trivial representation. For $k=1$, $R(S^2,1)$ is empty since $\pi(S^2\setminus\{a_1\})=1$ and hence admits no traceless representations. The space $R(S^2,2)$ is $S^2_{\bbi}/SU(2)$, which consists of a single point. Every $\rho\in R(S^2,3)$ is conjugate to the representation 
$x_1\mapsto \bbi, x_2\mapsto\bbj, x_3\mapsto -\bbk$.  Thus $R(S^2,3)$ consists of a single point. It is well known that $R(S^2,4)$ is the {\em  pillowcase}, the quotient of the 2-torus by the hyperelliptic involution, a space homeomorphic to the 2-sphere. For one proof see \cite{HHK1}.   

For $k=5$, it is known (see e.g. \cite{seidel4}) that $R(S^2,5)$ is homeomorphic to $\CC P^2\#5\overline{\CC P}^2$.  This can be seen from 
\cite[Section 4]{KK1} by observing first that Proposition \ref {prop1.1} below (specifically, the submersivity of $f$ and the fact that $R(S^2,5)_{\rm ab}$ is empty) implies  that $R(S^2,5)$ is diffeomorphic to $R_\ep(S^2,5)$, where $R_\ep(S^2,5)$ is defined by requiring representations to be traceless on $x_1, x_3, x_5$ but have  trace $-\ep$ for some small $\ep>0$ on $x_2, x_4$. Since $R(S^2,5)$ is smooth, $R(S^2,5)$ and $R_\ep(S^2,5)$ are diffeomorphic for small enough $\ep$, and the resulting configuration space of linkages in $S^3$  corresponds to the 12th polygon in \cite[Figure 7]{KK1}. Hence, by the results of that article, $R(S^2,5)$ is homeomorphic to $\CC P^2\#5\overline{\CC P}^2$.

 \medskip
   
\begin{df}\label{abloc}
Define the {\em abelian locus}   $R(S^2,k)_{\rm ab}\subset R(S^2,k)$ to
 consist of conjugacy classes of representations with abelian image. 
 \end{df}
 
 When $k$ is even,  $R(S^2,k)_{\rm ab}$ is a finite set containing $2^{k-2}$ points. It is  indexed by $\ep_2,\cdots,\ep_{k-2}\in\{\pm 1\}$, corresponding to the representations
$$x_1\mapsto \bbi,~ x_\ell\mapsto \ep_\ell\bbi \text{ for 1} <\ell<n, x_k\mapsto \ep_2\ep_3\dots\ep_{n-1}\bbi.$$ 
When $k$ is odd, $R(S^2,k)_{\rm ab}$ is empty, since a product of $k$ commuting traceless elements cannot equal 1.

 Define the polynomial  map 
  \begin{equation}\label{themapf}
f:(S^2_{\bbi})^{k-1}\to \RR, ~f(q_1,\dots, q_{k-1})=\Real(q_1\dots q_{k-1})
\end{equation}
 The map $f$ is invariant with respect to the diagonal $SU(2)$ conjugation action.

For convenience, denote $SU(2)/\{\pm 1\}$ by $PU(2)$. The conjugation action of $SU(2)$ on itself, and hence on representations, factors through $PU(2)$.

Let $ A\subset (S^2_{\bbi})^{k-1}$ denote the   set  
$$A=\{(\pm q,\dots,\pm q)~|~ q\in S^2_\bbi\}.$$ 
Thus $A$ is a disjoint union of $2^{k-2}$ 2-spheres, and $A/PU(2)$ is a collection of $2^{k-2}$ points.

Much of the following proposition is known; see e.g. \cite{Lin, Heusener2}. We include its simple proof for completeness.  
\begin{prop} \label{prop1.1} \hfill
 \begin{enumerate}

\item  The conjugation action of $PU(2) $ on $(S^2_{\bbi})^{k-1}\setminus A$ is free.
\item The restriction of the map $f$ of Equation (\ref{themapf}) to $(S^2_{\bbi})^{k-1}\setminus A$ is a submersion. 

  \item
The map 
$$\Phi: R(S^2,k)\to (S^2)^{k-1}/PU(2), ~\Phi(\rho)=[\rho(x_1),\dots, \rho(x_{k-1})]$$
induces an homeomorphism  $$R(S^2,k)\cong  f^{-1}(0)/PU(2) $$ which takes $R(S^2,k)_{\rm ab}$ to $A/PU(2)$. 

\item 
$R(S^2,k)\setminus R(S^2,k)_{\rm ab}$ is   smooth  of dimension $2k-6$.\end{enumerate}
\end{prop}
\begin{proof} If $(q_1,\dots,q_{k-1})\in (S^2_{\bbi})^{k-1} \setminus A$, there exists a pair $q_i,q_j$ which do not commute. The intersection of their stabilizers is therefore $\{\pm 1\}$, which proves the first assertion. 

For the second assertion, fix $(q_1,\cdots, q_{k-1})\in f^{-1}(0)\setminus A$ and choose $\ell$ so that  $q_\ell\ne \pm q_{\ell+1}$. Let $x= q_{\ell+2}q_{\ell+3}\dots q_{k-1}q_1q_2\dots q_{\ell-1}$.
At least one of  $ q_\ell x$ and $q_{\ell+1}x$ is different from $\pm1$.

If $q_\ell x\ne \pm1$, then $xq_\ell\ne \pm1$, so that $xq_\ell=e^{\alpha R}$ for $R\in S^2_\bbi$ and $\sin\alpha\ne 0$. Then 
$$\begin{multlined} 0=\Real(q_1\dots q_{k-1})=\Real(q_{\ell+1}xq_{\ell})=\Real(q_{\ell+1}e^{\alpha R})=\cos\alpha \Real(q_{\ell+1})+\sin\alpha \Real(q_{\ell+1}R) \\ =\sin\alpha \Real(q_{\ell+1}R).\hskip1.2in\end{multlined}$$
 Define $q_{\ell+1}(t)=q_{\ell+1}e^{tR}$.  Then $\Real(q_{\ell+1}(t))=\cos t \Real(q_{\ell+1})+\sin t \Real(q_{\ell+1}R)=0$, so that $q_{\ell+1}(t)\in S^2_{\bbi}$. Moreover
 $$\begin{multlined} \Real(q_1\dots q_{\ell}q_{\ell+1}(t) q_{\ell+2}\dots q_{k-1})=
 \sin t \Real(q_1\dots q_{\ell}R q_{\ell+2}\dots q_{k-1})\\ =\sin t \Real(xq_{\ell}R) =\sin t \sin\alpha\Real(R^2)=-\sin t \sin \alpha.
\end{multlined} $$
 Hence $\tfrac{d}{dt}|_{t=0}f(q_1,\dots, q_{\ell+1}(t),\dots , q_{k-1})=-\sin\alpha\ne 0.$

 If $q_{\ell+1}x\ne \pm1$, then write $q_{\ell+1}x=e^{\alpha R}$ with $\sin\alpha\ne 0$.  Define $q_{\ell}(t)$ to be $q_\ell e^{tR}$.  
 A similar computation shows that $q_{\ell}(t)\in S^2_\bbi$ and $\tfrac{d}{dt}|_{t=0}f(q_1,\dots, q_{\ell}(t),\dots , q_{k-1})=-\sin\alpha\ne 0.$
 
The third assertion follows from the fact that if $f(q_1,\dots, q_{k-1})=0$, then the homomorphism $\pi_1(S^2\setminus \{a_i\})\to SU(2)$ sending $x_i$ to $q_i$ for $i<k$ and $x_k$ to $(q_1\dots q_{k-1})^{-1}$ lies in $R(S^2,k)$. Moreover, it lies in $R(S^2,k)_{\rm ab}$ if and only if the $q_i$ all commute, which happens exactly when $(q_1,\dots, q_{k-1})\in A$.

The fourth assertion follows from the first three and the implicit function theorem

\end{proof}

When $k$ is odd, $R(S^2,k)_{\rm ab}$ is empty and hence $R(S^2,k)$ is   smooth. Since $R(S^2,2n+1)$ is smooth, its diffeomorphism type is unchanged by small perturbations of the traceless condition. More precisely, changing the conditions $\Real(\rho(x_i))=0$ to $\Real(\rho(x_i))=\delta_i$ for some small non-zero $\delta_i$ does not change the character variety, since the conjugacy class defined by the condition $\Real=\delta$ is a 2-sphere close to $S^2_\bbi$ and the submersivity of $f$ is a stable property. The relation $x_1\cdots x_{2n+1}=1$ shows that these character varieties can be identified with the configuration space of linkages of length $r_i$ for some generic $r_i$ close to $\frac\pi 2$, as in \cite{KK1}.  
Theorem 3.1 of \cite{KK1} then implies the following. 
\begin{prop}\label{linkages}  $R(S^2, 2n+1)$ is a smooth $2n-6$ manifold which admits a Morse function with only even index critical points.  \qed
\end{prop}

\bigskip

 For $k$ even, $R(S^2,2n)$ has a   $4n-6$ dimensional smooth locus and a finite singular locus $R(S^2,2n)_{\rm ab}$ consisting of $2^{2n-2}$ points (see e.g. \cite{HK}).

\begin{thm}\label{thm1.2} Each point in $R(S^2,2n)_{\rm ab}$ has a neighborhood in $R(S^2,2n)$ homeomorphic to a cone on $(S^{2n-3}\times S^{2n-3})/S^1$, for some free action of  $S^1$ on $S^{2n-3}\times S^{2n-3}$. \end{thm}
Theorem \ref{thm1.2} can be derived from Theorem A.7 in the unpublished article \cite{JR} which considers the representation variety rather than the character variety.
For completeness and because we need the details of the construction below, we give a proof of Theorem \ref{thm1.2} in Appendix \ref{apB}.  Our approach is similar to  the argument of \cite{JR}.  Our proof is shorter, as we use some tricks to streamline the verification of the non-degeneracy of the Hessian and to calculate its signature.

\medskip

\begin{df}\label{bdloc}
Define the {\em binary dihedral locus} 
   $R(S^2,k)_{\rm B}\subset R(S^2,k)$ to
 consist of conjugacy classes representations conjugate to representations with image in the binary dihedral  subgroup 
$$B=\{e^{\theta\bbk}\}\cup \{e^{\theta\bbk}\bbi\}\subset SU(2)$$
and which, in addition, take each $x_\ell$ into the coset $\{e^{\theta\bbi}\bbj\}$. Since $\Real(e^{\theta\bbk}\bbi)=0$, $R(S^2,k)_{\rm B}\subset R(S^2,k)$.
\end{df}

We emphasize that, with this definition, Image$(\rho)\subset B$ does not imply that $\rho\in R(S^2,k)_B$. For example, the representation $\rho\in R(S^2,3)$ taking $x_1$ to $\bbi$, $x_2$ to $\bbj$ and $x_3$ to $-\bbk$ has image in $B$ but does not lie in $R(S^2,3)_B$ since $-\bbk\ne e^{\theta\bbk}\bbi$ for any $\theta$.

  When $k$ is odd, $R(S^2,k)_{B}$ is empty. This is because product of an odd number of elements of the index 2 coset $\{e^{\theta\bbk}\bbi\}$  lies again in this coset and hence cannot equal the identity.
Note that  $$R(S^2,2n)_{\rm ab}\subset R(S^2,2n)_{B}.$$  

\medskip
Let $\TT^\ell$ denote the $\ell$-torus, $\TT^\ell=(S^1)^\ell$.

\begin{prop}\label{prop2.6} The binary dihedral subspace $R(S^2,2n)_{B}$    homeomorphic to $\TT^{2n-2}/(\ZZ/2)$, where $\ZZ/2$ acts on the $(2n-2)$-torus  by $(e^{\alpha_1\bbi},\dots, e^{\alpha_{2n-2}\bbi})\mapsto (e^{-\alpha_1\bbi},\dots, e^{-\alpha_{2n-2}\bbi}) $.  The homeomorphism identifies the fixed points of the $\ZZ/2$ action with the finite abelian subspace $R(S^2,2n)_{\rm ab}\subset R(S^2,2n)_{B}$. Hence a neighborhood of $\rho\in R(S^2,2n)_{\rm ab}$ in $R(S^2, 2n)_B$ is homeomorphic to a cone on $\RR P^{2n-3}$.
\end{prop}
\begin{proof} Since $e^{-(\theta/2)\bbi}e^{\theta\bbk}\bbi e^{(\theta/2)\bbi}=\bbi$, any binary dihedral representation $\rho$ can be conjugated 
so that $\rho(x_1)=\bbi$ and $\rho(x_\ell)=e^{\theta_\ell\bbk}\bbi$ for $\ell>1$. The relation $x_1\dots x_{2n}=1$ implies that 
$$(-1)^n e^{(-\theta_2+\theta_3+\dots-\theta_{2n})\bbk}=1,$$
so that $\rho(x_{2n})=e^{(n\pi-\theta_2+\theta_3-\dots+\theta_{2n-1})\bbk}$. Taking $\theta_i=\alpha_{i+1}$ for $i=2,\dots,n-1$ and  $\theta_{2n}= n\pi-\theta_2+\theta_3-\dots+\theta_{2n-1}$ therefore gives a surjection $\TT^{2n-2}\to R(S^2,2n)_{B}$.    Conjugation by $\bbi$ preserves $B$ and takes $e^{\theta\bbk}\bbi$ to $e^{-\theta\bbk}\bbi$. It follows that this surjection factors through the $\ZZ/2$ action, and it is easy to see that  the resulting map $\TT^{2n-2}/(\ZZ/2)\to R(S^2,2n)_B$ is injective, and hence a homeomorphism.

The fixed points of the $\ZZ/2$ action on $\TT^{2n-2}$ are   the $2^{2n-2}$ points $(\pm 1,\dots,\pm 1)$. These correspond via $\TT^{2n-2}\to R(S^2, 2n)_B$ precisely to the abelian representations $\rho(x_i)=\pm \bbi$. 
\end{proof}

\section{2-fold covers}

Consider the homomorphism 
\begin{equation}\label{defofalpha}\alpha:\pi_1(S^2\setminus\{a_i\}_{i=1}^{k})\to \{\pm 1\}, ~ \alpha(x_i)=-1.
\end{equation}
Note that this only makes sense when $k$ is even, so we asume that $k=2n$.  Then $\alpha$ defines a 2-fold cyclic branched cover $p:F\to S^2$  branched over $\{a_i\}_{i=1}^{2n}$. The Riemann-Hurwitz formula identifies $F$ as an orientable closed surface of genus $n-1$, and so we denote it by $F_{n-1}$. Denote the preimage $p^{-1}(a_i)$ by $\ta_i$.

Identify $\pi_1(F_{n-1}\setminus \{\ta_i\})$ with $\ker\alpha$  via the induced injection $$p_*:\pi_1(F_{n-1}\setminus \{\ta_i\})\to \pi_1(S^2\setminus \{a_i\}).$$  Notice that $x_i^2\in \pi_1(F_{n-1}\setminus \{\ta_i\})$ and in fact represents  the meridian  of $\ta_i$ in $F_{n-1}\setminus \{\ta_i\}$.

\medskip

 Recall that  $R(F_{n-1})$ denotes the $SU(2)$ character variety of the closed surface $F_{n-1}$.
 Denote by $R(F_{n-1})_{\rm ab}\subset R(F_{n-1})$ the subvariety of representations with abelian image.
 
\begin{df} \label{othertrace}
 Let $R(F_{n-1}, 2n)_{-1}$ and $R(F_{n-1}, 2n)_{1}$ be defined as 
$$R(F_{n-1}, 2n)_{\pm 1}=\{\rho\in \Hom(\pi_1(F_{n-1}\setminus \{\ta_i\}),SU(2)) ~|~ 
 \rho( x_i^2) =\pm 1\}/_{\rm conjugation}.$$
 \end{df}

 \medskip

Restricting to $\ker \alpha$  induces a well defined function $$p^*:R(S^2,2n)\to R(F_{n-1},2n)_{-1}.$$ This is because if $z\in SU(2)$ satisisfies $\Real(z)=0$, then $z^2=-1$, hence $p^*(\rho)(x_i^2)=-1$.

The inclusion $i:F_{n-1}\setminus \{\ta_i\}\subset F_{n-1}$ induces a function 
$$i^*: R(F_{n-1})\to R(F_{n-1},2n)_1$$
since   the meridians of $F_{n-1}\setminus \{\ta_i\}$ are nullhomotopic in $F_{n-1}$.
 
Let $\rho_0\in R(S^2,2n)_{\rm ab}$ denote the (conjugacy class of  the) representation sending every $x_i$ to $\bbi$. Then define
$c\in R(F_{n-1}, 2n)_{-1}$ by $c=p^*(\rho_0)$. Since $\rho_0$ sends every loop in $\pi_1(S^2\setminus\{a_i\})$ to $\{\pm 1,\pm\bbi\}$ it follows that $c$ takes values in the center $\{\pm 1\}\subset SU(2)$, and takes the value $-1$ on each meridian of $F_{n-1}\setminus \{\ta_i\}$.

\begin{prop}\label{prop2.1}\hfill
\begin{enumerate}
\item Pointwise multiplication by $c$ induces a homeomorphism 
$$c_*:R(F_{n-1}, 2n)_{-1}\to R(F_{n-1},2n)_1.$$  
\item The restriction $i^*: R(F_{n-1})\to R(F_{n-1},2n)_1$ is a homeomorphism.
\end{enumerate} 
\end{prop}

\begin{proof}
 The pointwise product of a representation with a central representation is again a representation.   Thus $c$ induces maps 
 $R(F_{n-1}, 2n)_{-1}\to R(F_{n-1},2n)_1$ and $R(F_{n-1}, 2n)_{1}\to R(F_{n-1},2n)_{-1}$; these maps are inverses of each other since $c^2=1$. The second statement follows from the Seifert-Van kampen theorem.  
   \end{proof}

Denote by $\tp$ the composite
\begin{equation}
\label{eq3}\tp=i^*c_*p^*:R(S^2,2n)\to  R(F_{n-1})
\end{equation}

Pointwise multiplication by $\alpha$, defined in Equation (\ref{defofalpha}), determines a homeomorphism of order 2: \begin{equation}
\label{eq4} \alpha_*:R(S^2,2n)\to R(S^2,2n).
\end{equation}

Some of the statements of the following theorem are known to experts, for example see  \cite{FKP, PS}.  We include its   proof  in Appendix \ref{apC} for convenience.

\begin{thm}\label{prop2.2}
 \hfill
\begin{enumerate}
\item The fixed point set of  the involution $\alpha_*$ is $R(S^2,2n)_B$.
\item $\tp:R(S^2,2n)\to  R(F_{n-1})$ induces an injection  
$R(S^2,2n)/\langle \alpha_*\rangle\subset  R(F_{n-1})$.
\item The image of $\tp:R(S^2,2n)_B\to  R(F_{n-1})$ is precisely $ R(F_{n-1})_{\rm ab}$.
\end{enumerate} 
\end{thm}

\section{The traceless character variety $R(S^2,6)$}

  For a closed surface $F_g$ of genus $g$, it is known (see e.g. \cite{AB,goldman1}) that  
\begin{equation}\label{surfde}
R(F_g)=R(F_g)_{\rm cen}\sqcup R(F_g)^*_{\rm ab}\sqcup R(F_g)_{\rm irr}
\end{equation}
where $R(F_g)_{\rm cen}$ denotes the conjugacy classes of central representations (those with value in the center $\{\pm1\}\subset SU(2)$),  $R(F_g)^*_{\rm ab}=R(F_g)_{\rm ab}\setminus R(F_g)^*_{\rm cen}$ and $R(F_g)_{\rm irr}=R(F_g)\setminus R(F_g)_{\rm ab}$.  The central locus $R(F_g)_{\rm cen}$ is a finite set of $2^{2g}$ points, the abelian, non-central locus $R(F_g)^*_{\rm ab}$ is a smooth semi-algebraic set of dimension $ 2g$, and the irreducible locus $R(F_g)_{\rm irr}$ is smooth of dimension $6g-6$. 

Theorem \ref{prop2.2} can be summarized in the diagram 
  \begin{equation*}
\setcounter{equation}{7}
\label{diag1} \includegraphics{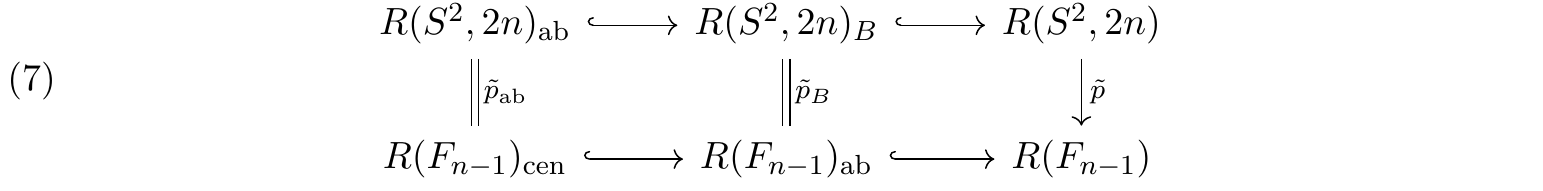}
\end{equation*}
with the restrictions $\tp_{\rm ab}, \tp_B$ homeomorphisms, and the right vertical $\tp$ the  quotient of the involution $\alpha_*$ onto its image.  
When $n=1$ all the spaces in the diagram contain only one point. When $n=2$ the upper and lower   right inclusions are   equalities and all vertical maps homeomorphisms. Indeed, it is known (see e.g. \cite{Heusener-Kroll,Lin} and \cite[Proposition 3.1]{HHK1}) that $ R(S^2,4)=R(S^2,4)_B$, so that $\alpha_*$ acts trivially on $ R(S^2,4)$. Moreover since a genus 1 surface has abelian fundamental group, $R(F_{n-1})_{\rm ab}=R(F_{n-1})$. As mentioned above,  $R(S^2,4)$ and $ R(F_1)$ are homeomorphic   to the {\em pillowcase}, a topological 2-sphere obtained as the quotient of the hyperelliptic involution on the torus.    

 When $n>2$,  the dimensions of the  top strata of the semi-algebraic sets on the top line of (7) are (using  Theorem  \ref{thm1.2} and Proposition \ref{prop2.6}) $0, 2n-2,
 4n-3$  and in the bottom line $0, 2n-2, 6n-12$.    
Thus for dimensional reasons, when $n>3$,  $\tp:R(S^2,2n)\to R(F_{n-1})$ cannot be surjective, and hence $\tp:R(S^2;2n)/\langle \alpha_*\rangle\subset  R(F_{n-1})$ cannot be a homeomorphism.

\medskip

Thus the interesting case is when $n=3$.  The   following theorem shows that $\tp:R(S^2,6)\to R(F_2)$ is a surjective, and hence $R(F_2)$ is the orbit space of $R(S^2,6)$ by the  involution $\alpha_*$ with fixed point set $R(S^2,6)_B$.

\begin{thm} \label{main}The map  $\tp: R(S^2, 6)\to  R(F_2)$ is surjective.
\end{thm}

The proof of Theorem \ref{main} relies on the following lemma.
 
\begin{lem}\label{lem5.2} Suppose that $a,b,c,d\in SU(2)$ satisfy $abcd=dcba$. Then there exists $x\in SU(2)$ satisfying 
$$0=\Real(x)=  \Real(xa)=  \Real(xb)= \Real(xc)= \Real(xd)=  \Real(x(abcd)^{-1})
 $$
\end{lem}
\begin{proof}  A lengthy but uninteresting calculation shows that if $a',b',c',d'\in SU(2)$ satisfy $0=\Real(\bbi a')=\Real(\bbi b')=\Real(\bbi c')=\Real(\bbi d')$, then $\Real(\bbi a'b'c'd')=-\Real(\bbi d'c'b'a')$. Conjugating by  $g\in SU(2)$ one may replace $\bbi$ by any purely imaginary quaternion $x$ and reach the same conclusion. Therefore, if $x\in S^2_\bbi$ satisfies  
 \begin{equation}
\label{ugly}
0=\Real(x)=  \Real(xa)=  \Real(xb)= \Real(xc)= \Real(xd),
\end{equation}
then $\Real(xdcba)=\Real(xabcd)=-\Real(xdcba)$ and hence
$\Real(xabcd)=0$, which, since $\Real(x)=0$, implies $\Real(x(abcd)^{-1})=0$.  Hence it suffices to find an $x$ satisfying Equation (\ref{ugly}).

Suppose first that $[a,b]\ne 1$. Since  $\Real(ab-ba)=0$,  $x=\frac{ab-ba}{\| ab-ba\|}$ is a unit quaternion satisfying $\Real(x)=0$. The identities
$$d^{-1}(abc)d=cba \text{ and } c^{-1}(d^{-1}ab)c=bad^{-1}$$
follow from $abcd=dcba$, and hence $\Real(abc)=\Real(cba)$ and $\Real(d^{-1}ab)=\Real(bad^{-1})$.  Therefore $\Real((ab-ba)c)=0$, which implies $\Real(xc)=0$. Similarly $\Real(xd^{-1})=0$, and so $\Real(x)=0$ implies that $\Real(xd)=0$. Next, $\Real((ab-ba)a)=\Real(aba-ba^2)=0$ so that $\Real(xa)=0$; similarly $\Real(xb)=0$.     Thus $x$ satisfies Equation (\ref{ugly}), as desired.

If $[a,b]=1$ but $[b,c]\ne 1$, then repeat the argument, taking $x=\frac{bc-cb}{\| bc-cb\|}$, using the identities $a(bcd)a^{-1}=dbc$ and $d^{-1}(abc)d=cba$.

If $1=[a,b]=[b,c]=1$ but $[c,d]\ne 1$, again repeat the argument, taking $x=\frac{cd-dc}{\| cd-dc\|}$. Next, if $1=[a,b]=[b,c]=[c,d]$ but $[d,a]\ne 1$, then repeat, taking $x=\frac{da-ad}{\| da-ad\|}$.

If $1=[a,b]=[b,c]=[c,d]=[d,a]$, then $bacd=abcd=dcba=dcab$, so repeating the steps above we can find an  $x$ satisfying Equation (\ref{ugly}) provided $[a,c]\ne1$ or $[b,d]\ne 1$.

If $a,b,c,d$ pairwise commute, then they all lie in the same circle subgroup $\{e^{\theta Q}\}$ for some $Q$ a purely imaginary unit quaternion.  Choose $x$  satisfying $\Real(x)=0$ and $\Real(xQ)=0$ (for example, if $Q=g\bbi g^{-1},$ take $x=  g\bbj g^{-1}$). Then $\Real(e^{\theta Q}x)=0$ for all $\theta$, so that $x$ satisfies Equation (\ref{ugly}), completing the proof of Lemma \ref{lem5.2}.
\end{proof}

\begin{proof}[Proof of Theorem \ref{main}]
 Figure \ref{fig1} shows ten curves on a 6-punctured genus 2 surface $F_2\setminus \{\ta_i\}_{i=1}^6$. Rotation  of angle $\pi$ about the horizontal axis is the deck transformation of the cover $p:F_2\setminus \{\ta_i\}\to S^2\setminus \{a_i\}$.

\begin{center}
\begin{figure}[!ht]
\def\svgwidth{5in} 
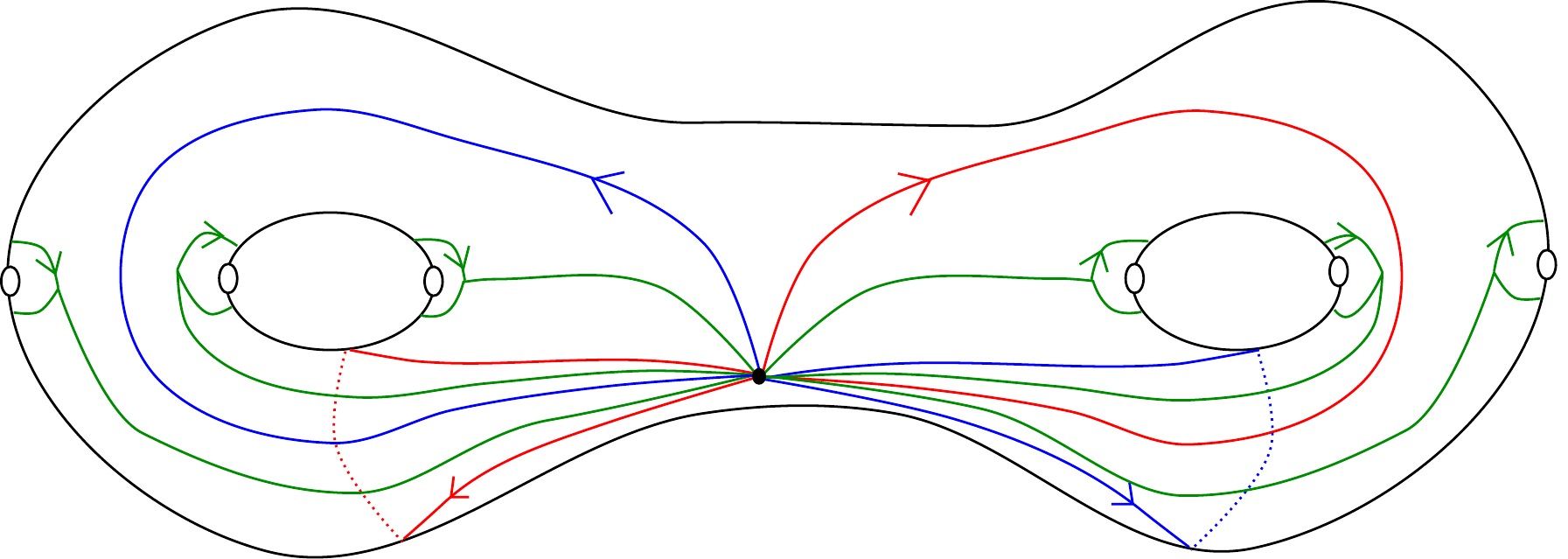
\caption{\label{fig1}}
\end{figure}
\end{center}

From the figure one reads off the presentation: 
$$\pi_2:= \pi_1(F_{2}\setminus \{\ta_i\})=\langle r_1,s_1,r_2,s_2,y_1,\dots,y_6~|~
 r_1y_3s_1y_2r_1^{-1}y_1s_1^{-1}r_2y_6 s_2y_5r_2^{-1}y_4 s_2^{-1}=1\rangle. $$
 The monomorphism to
 $$\pi:= \pi_1(S^2\setminus \{a_i\})=\langle x_1,\dots,x_6~|~
x_2x_2x_3x_4x_5x_6=1\rangle $$
induced by $p$ is given by 
$$p_*(y_i)=x_i^2, ~p_*(r_1)=x_1x_2, ~p_*(s_1)=x_3^{-1}x_2^{-1},~p_*(r_2)=x_4x_5,~ p_*(s_2)=x_6^{-1}x_5^{-1}.$$

Given $\rho:\pi_2\to SU(2)$ satisfying $\rho(y_i)=-1$, the task is to extend $\rho$ to $\pi$. For clarity of exposition we abuse notation and identify generators of $\pi_2$ with their image under $\rho$. Thus  we are given $r_i,s_i, y_i\in SU(2)$ satisfying the relation for $\pi_2$ and in addition $y_i=-1$, and we seek $x_i\in SU(2)$  satisfying
\begin{equation}
\label{require}
\Real(x_i)=0,~ r_1=x_1x_2,~s_1=x_3^{-1}x_2^{-1},~r_2=x_4x_5,~ s_2=x_6^{-1}x_5^{-1},\text{ and } x_1\dots x_6=1
\end{equation}
(recall that $x_i^2=-1$ is equivalent to $\Real(x_i)=0$).
Since $y_i=-1$ and $-1$ is central in $SU(2)$, the relation in the presentation of $\pi_2$ implies that the surface relation $$[r_1,s_1][r_2,s_2]=1$$ holds.

\medskip

Define $a,b,c,d,e$ by 
\begin{equation}\label{abcde} a=r_1, ~b=s_1^{-1}r_1^{-1}, ~c=s_2s_1, ~d=s_1^{-1}r_2s_2^{-1}, ~ e=r_2^{-1}s_1.\end{equation}
Then 
$$
cdeab=s_2s_1s_1^{-1}r_2s_2^{-1}r_2^{-1}s_1r_1s_1^{-1}r_1^{-1}=[s_2,r_2][s_1,r_1]=([r_1,s_1][r_2,s_2])^{-1}=1
$$
and 
$$edcba=r_2^{-1}s_1s_1^{-1}r_2s_2^{-1}s_2s_1s_1^{-1}r_1^{-1}r_1=1$$
and hence 
$$
e^{-1}=abcd=dcba
.$$
Lemma \ref{lem5.2} implies that  there exists an $x_1\in SU(2)$ satisfying 
$$0=\Real(x_1)=  \Real(x_1a)=  \Real(x_1b)= \Real(x_1c)= \Real(x_1d)=  \Real(x_1e).
$$
The assignment
 \begin{equation}\label{exes} x_2=x_1^{-1}r_1, ~ x_3=r_1^{-1}x_1s_1^{-1},~ x_4=s_1x_1^{-1}s_2,~ x_5=s_2^{-1}x_1s_1^{-1}r_2,~ x_6=r_2^{-1}s_1x_1^{-1}\end{equation}
is easily verified to satisfy (\ref{require}).  For example 
$$\Real(x_4)=\Real(s_1x_1^{-1}s_2)=\Real(x_1^{-1}s_2s_1)=-\Real(x_1c)=0,$$
$$x_6^{-1}x_5^{-1}=x_1s_1^{-1}r_2r_2^{-1}s_1x_1^{-1}s_2=s_2,
$$
and 
$$x_1x_2x_3x_4x_5x_6=x_1
x_1^{-1}r_1r_1^{-1}x_1s_1^{-1} s_1x_1^{-1}s_2s_2^{-1}x_1s_1^{-1}r_2r_2^{-1}s_1x_1^{-1}=1.$$
We leave the rest of the verifications, which complete the proof of Theorem \ref{main}, to the reader.
\end{proof}

It is worth emphasizing that the proofs of Theorem \ref{main}  and Lemma \ref{lem5.2}, and in particular Equations (\ref{abcde}) and (\ref{exes}), explicitly  give  the    extension of $\rho:\pi_1(F_2)\to SU(2)$ to a traceless representation of $\pi_1(S^2\setminus \{a_i\}_{i=1}^6)$.	We state this formally.

\begin{cor}\label{extend} Let $\rho:\pi_1(F_2)=\langle r_1, s_1, r_2, s_2~|~[r_1,s_1][r_2,s_2]\rangle\to SU(2)$ be a representation of the fundamental group of a closed surface of genus 2. Define 
$$a=\rho(r_1),  ~b=\rho(s_1^{-1}r_1^{-1}), ~c=\rho(s_2s_1), ~d=\rho(s_1^{-1}r_2s_2^{-1}), ~ e=\rho(r_2^{-1}s_1).$$
Fix a sign $\ep\in \{\pm 1\}$. If $[a,b]\ne 1$, define
$\rho_\ep(x_1)=\ep\frac{
b-ba}{\|ab-ba\|}$. If $[a,b]=1$ but $[b,c]\ne 1$, define $\rho_\ep(x_1)=\ep \frac{bc-cb}{\|bc-cb\|}$, and so forth, going through the ordered list of commutators $[a,b], [b,c], [c,d], [d,a], [a,c], [b,d]$. If all of these commutators are trivial, so that $a,b,c,d$ lie in a subgroup $\{e^{\alpha Q}\}$, then define $\rho_\ep(x_1)$ to be $P$ for any $P\in S^2_\bbi$ satisfying $\Real(PQ)=0$. 
Set 
$$\begin{multlined}\rho_\ep( x_2)=\rho_\ep(x_1)^{-1}\rho(r_1), ~ \rho_\ep(x_3)=\rho(r_1)^{-1}\rho_\ep(x_1)\rho(s_1^{-1}),~ \rho_\ep(x_4)=\rho(s_1)\rho_\ep(x_1)^{-1}\rho(s_2),~\\ \rho_\ep(x_5)=\rho(s_2^{-1})\rho_\ep(x_1)\rho(s_1^{-1}r_2),~ \rho_\ep(x_6)=\rho(r_2^{-1}s_1)\rho_\ep(x_1)^{-1}.\hskip1in
\end{multlined}$$

Then $\rho_{\pm 1}\in R(S^2,6)$  and  $\tilde p(\rho_\pm)=\rho$.  The two choices $\rho_{\pm 1}$ represent the two conjugacy classes in the fiber of  when $\rho$ is non-abelian, and give conjugate binary dihedral representations when $\rho$ is abelian.\qed
\end{cor}

\medskip

We turn now to a more precise identification of the singularities of $R(S^2,6)$ and to a local description of the map $\tilde p: R(S^2,6)\to R(F_2)$ near the singularities. 

We first recall some aspects of the proof of Theorem \ref{thm1.2}. The cone neighborhoods of each point in $R(S^2,6)_{\rm ab}$ are all identical. Let
$\rho_0\in R(S^2,6)_{\rm ab}$ denote the representation which sends $x_\ell$ to $\bbi$, $\ell=1,\dots,6$.  
 Equation (\ref{cutout}), rewritten in complex coordinates,   shows that  a neighborhood of $\rho_0$ in $R(S^2,6)$ is homeomorphic to $h^{-1}(0)/S^1$, where 
\begin{equation}\label{cut}
h:\CC^4\to \RR,~ h(z_1,z_2,z_3,z_4)= \Real(\bbi e^{-z_1\bbj}e^{z_2\bbj}e^{-z_3\bbj}e^{z_4\bbj}).
\end{equation}
The correspondence $h^{-1}(0)/S^1\cong {\rm nbd}(\rho_0)$ 
 takes $(z_1,z_2,z_3,z_4)\in h^{-1}(0)$ to the representation
\begin{equation}\label{corr}
x_1\mapsto \bbi,~x_2\mapsto e^{z_1\bbj}\bbi,~
x_3\mapsto e^{z_2\bbj}\bbi,~ x_4\mapsto e^{z_3\bbj}\bbi,~x_5\mapsto e^{z_4\bbj}\bbi,~x_5\mapsto (\bbi e^{z_1\bbj}\bbi e^{z_2\bbj}\bbi e^{z_2\bbj}\bbi e^{z_3\bbj}\bbi e^{z_4\bbj}\bbi)^{-1}.\end{equation}

The $S^1$ action is given by $e^{\theta\bbi} (z_1,z_2,z_3,z_4)=(e^{2\theta\bbi}z_1,e^{2\theta\bbi}z_2,e^{2\theta\bbi}z_3,
e^{2\theta\bbi}z_4)$. For convenience identify $S^1/\pm1$ with $S^1$ and rewrite this action as $(e^{\theta\bbi}z_1,e^{ \theta\bbi}z_2,e^{ \theta\bbi}z_3,
e^{ \theta\bbi}z_4)$.  

  Let $\tau:\CC^4\to \CC^4$ denote complex conjugation in each coordinate.  Since  
$e^{\bar z\bbj}=\bbj e^{z\bbj}\bbj^{-1}$,
$$h(\bar z_1,\bar z_2, \bar z_3, \bar z_4)=\Real(\bbi e^{-\bar z_1\bbj}e^{\bar z_2\bbj}e^{-\bar z_3\bbj}e^{\bar z_4\bbj})=-h(z_1,z_2,z_3,z_4).
$$
In particular,  $\tau$ preserves $h^{-1}(0)$. Also, conjugating $-e^{z\bbj}\bbi$ by $\bbj$ yields $e^{\bar z\bbj}\bbi$, and hence
 $\tau$ corresponds, via (\ref{corr}), to the action of $\alpha_*$ on $R(S^2,6)$ (Equation (\ref{eq4})).

 The fixed points of $\tau$, namely $\RR^4\subset \CC^4$, lie in $h^{-1}(0)$  and correspond precisely to the binary dihedral representations.  Intersecting with the unit sphere $S^7\subset \CC^4$ gives 
 $$S^3=\RR^4\cap S^7\subset h^{-1}(0)\cap S^7\subset S^7$$
Away  from $0$, $\RR^4\subset \CC^4$ meets  each $S^1$ orbit in two points $\pm (r_1,r_2,r_3,r_4)$.   Hence taking the quotient by   $S^1$,and applying Theorem \ref{thm1.2}  yields the following links of a neighborhood of the singular point $\rho_0\in R(S^2, 6)_{\rm ab}$ in $R(S^2, 6)_{B}$ and $R(S^2, 6)$:
 \begin{equation*}
\setcounter{equation}{14}\label{flag}
\includegraphics{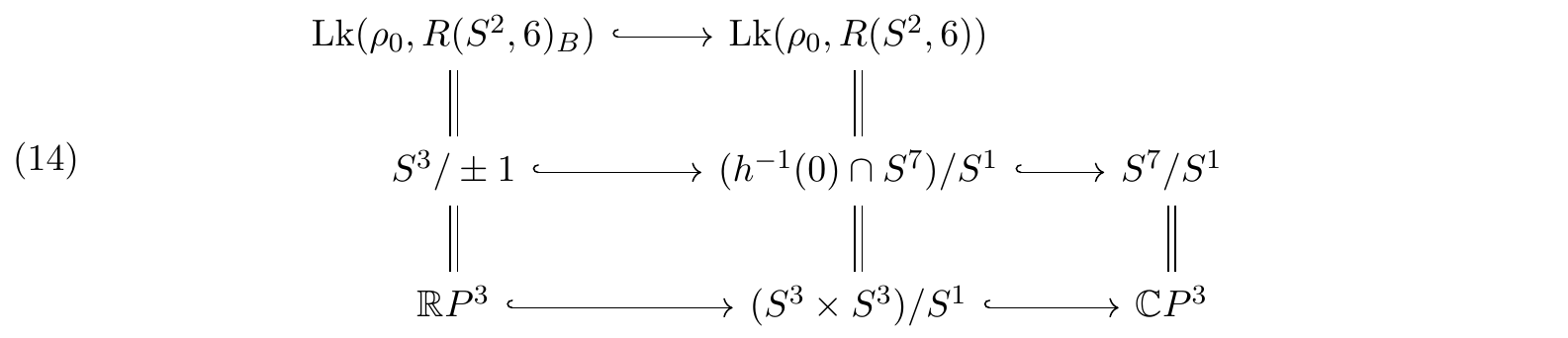}
\end{equation*}
with $\RR P^3\subset \CC P^3$ the fixed points of $\tau$.

\begin{thm}\label{link2} The link of the singular point $\rho_0\in R(S^2,6)_{ab}$  in $R(S^2,6)_B$ diffeomorphic to $\RR P^3$ and its link in  $R(S^2,6)$ is diffeomorphic to $S^2\times S^3$. Moreover, the link $S^2\times S^3$ embeds in $\CC P^3$, is invariant under complex conjugation, with fixed point set the link $\RR P^3$.
\end{thm}

\begin{proof} The only statement which is not covered by the discussion preceding the statement and Diagram (14) is the assertion that 
$(S^3\times S^3)/S^1$ is diffeomorphic to $S^2\times S^3$. 

The main result of \cite{GL} shows that there are only two inequivalent free $S^1$ actions on $S^3\times S^3$, distinguished by the second Stiefel-Whitney class of their orbit space.    Since the quotient $M=(S^3\times S^3)/S^1$ embeds in $\CC P^3$ with codimension 1, its normal bundle is trivial, and hence $w_2(M)=w_2(\CC P^3)|_M=0$. Thus \cite{GL} shows that $M\cong S^2\times S^3.$
\end{proof}

Referring to Diagram (7), one sees that the map $\tp:R(S^2, 6)\to R(F_2)$ induces a diagram of links of cone neighborhoods of $\rho_0$ and $\tp(\rho_0)$:
 \begin{equation*}
\label{flag2}\setcounter{equation}{15}
\includegraphics{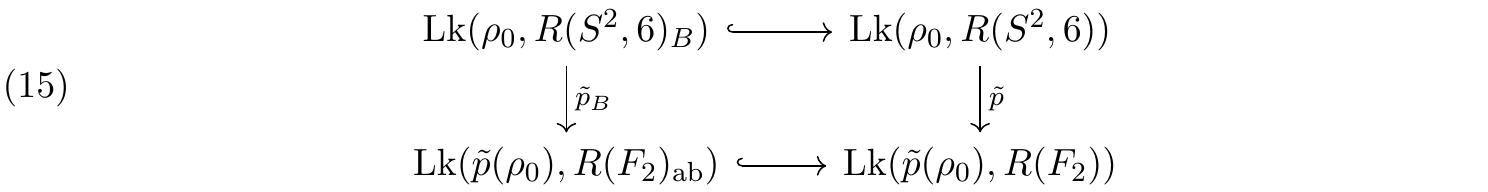}
\end{equation*}
 The left vertical map is a homeomorphism since $\alpha_*$ fixes the binary dihedral locus.  Hence $  {\rm Lk}(\tp(\rho_0), R(F_2)_{\rm ab}) =\RR P^3$.  Theorem \ref{main} implies that the right vertical map is the quotient by $\alpha_*$, and Theorem \ref{link2} identifies this map with the quotient
 $S^2\times S^3\to (S^2\times S^3)/(\ZZ/2)$,
 the action given by complex conjugation $\tau$ with fixed set $\RR P^3$. Hence Diagram (15) is, concretely:
 \begin{equation*}\label{quo}\setcounter{equation}{16}
 \includegraphics{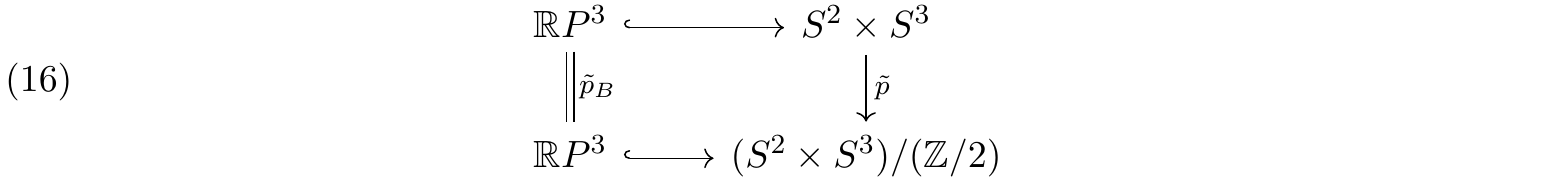}
\end{equation*}

Using results from algebraic geometry we can sharpen these identifications as follows. 
The Narasimhan-Seshadri theorem \cite{NS} states that $R(F_g)$ is homeomorphic to the moduli space $\mathcal{U}(2,0)_0$  of semi-stable rank 2, degree 0 holomorphic bundles over $F_g$  with trivial determinant.   
A theorem of Narasimhan-Ramanan \cite{nara}  states that $ \mathcal{U}(2,0)_0$ is homeomorphic to $\CC P^3$. Moreover,  in this identification,  $R(F_2)_{\rm ab}\subset R(F_2)$ corresponds to a singular Kummer surface $K$ with 16 nodal surface singularities, obtained as the quotient of the Jacobian of $F_2$ by an involution with 16 fixed points.

   A nodal surface singularity in $\CC^3$ takes the local form $x^2+y^2+z^2=0$; this is an isolated singularity with link $\RR P^3$.  The two fold branched cover of $\CC^3$ branched over a nodal surface singularity is the singular 3-fold in $\CC^4$ with equation $x^2+y^2+z^2+w^2=0$. The link of this singularity is diffeomorphic to 
 $S^2\times S^3$ (see e.g. \cite{Kauffman1})  Thus we recover  Diagram (16) from the Narasimhan-Ramanan theorem, and in addition conclude, since $\CC P^3$ is a 6-manifold, that 
 $$(S^2\times S^3)/(\ZZ/2)\cong S^5.$$
The Jacobian of $F_2$ is $H^1(F_2;\RR)/H^1(F_2,\ZZ)\cong \TT^4$.  Hence we recover from Diagram (7) the fact that $R(S^2,6)_B\cong  \TT^4/(\ZZ/2)$.

We summarize the previous discussion in the following theorem

\begin{thm}\label{thm4}  The diagram {\rm (7)} when $n=3$ is  equivalent to 
\begin{equation*}
\includegraphics{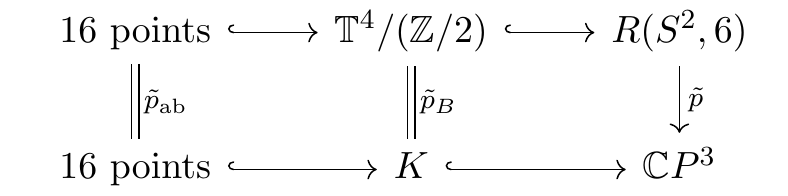}
\end{equation*}
with $\tp$ a 2-fold branched cover branched over the singular Kummer surface $K$ with 16 nodal singularities.

The diagram {\rm (15)} of links of any of the 16 singular points,   for $n=3$,  is  equivalent to 
\begin{equation*} 
\includegraphics{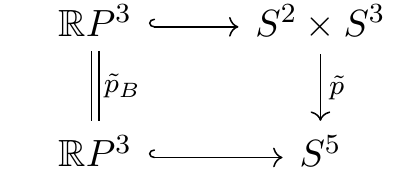}
\end{equation*}
\qed
\end{thm}

\appendix

\section{Proof of Theorem \ref{thm1.2}}\label{apB}

\begin{proof}[Proof of Theorem \ref{thm1.2}] The map $f:(S^2_\bbi)^{2n-1}\to \RR$ given by $f(q_1,\dots,q_{2n-1})=\Real(q_1\dots q_{2n-1})$ is $SU(2)$ invariant and satisfies $f^{-1}(0)/SU(2)\cong R(S^2,2n)$.
It is convenient to replace the map $f$ by a slightly simpler map
$$g:(S^2_\bbi)^{2n-2}\to \RR,~g(q_2,\dots,q_{2n-1})=\Real(\bbi q_2\dots q_{2n-1}).$$
Since every $q_1\in S^2_\bbi$ is conjugate to $\bbi$, and since the stabilizer  of $\bbi$  in the conjugation action of $SU(2)$ is the subgroup $S^1=\{e^{\bbi\theta}\}$,
$$g^{-1}(0)/S^1=f^{-1}(0)/SU(2)\cong R(S^2, 2n).$$

The circle $S^1$ acts  on $(S^2)^{2n-2}$ by conjugation on each factor: $$e^{\theta\bbi}(q_2,\dots,q_{2n-1})=(e^{\theta\bbi}q_2e^{-\theta\bbi},\dots,e^{\theta\bbi}q_{2n-1}e^{-\theta\bbi}).$$ Note that  $-1$ acts trivially, and that the induced $S^1/\{\pm\}$  action is free away from the set of  $2^{2n-2}$ fixed points $\Sigma=\{(\pm \bbi, \pm \bbi,\dots,\pm \bbi)\}$. These points correspond precisely to $R(S^2,2n)_{\rm ab}$.
  The function $g$ is invariant, and hence $S^1/\{\pm\}$ acts freely on $g^{-1}(0)\setminus \Sigma$. 

For any sequence of signs $\ep=(\ep_2, \dots, \ep_{2n-2}),~\ep_i\in \{\pm 1\}$, the map $$\ep_*:(S^2_\bbi)^{2n-2}\to (S^2_\bbi)^{2n-2},~\ep_*(q_2,\dots,q_{2n-2})= (\ep_2 q_2,\dots,\ep_{2n-2}q_{2n-2})$$ is an $S^1$ equivariant homeomorphism which preserves $g^{-1}(0)$ and  permutes the points of $\Sigma$. Hence it suffices to identify a neighborhood of $(\bbi,\dots, \bbi)$ in $g^{-1}(0)/S^1$; all the other singular points will have homeomorphic neighborhoods. 

The (exponential) map $e:\RR^2\to \HH$ given by $e(x+y\bbi)= \bbi e^{x\bbj+y\bbk}$ takes values in $S^2_\bbi$, since if $x+y\bbi=re^{\theta\bbi}$, $e(x,y)= \bbi e^{re^{\theta\bbi}\bbj}=\cos r\bbi +\sin r e^{\theta\bbi}\bbj\in S^2_\bbi$. It is easy to check that $e$ is a local diffeomorphism near $0$, and, since 
$$e^{e^{2\theta\bbi}(x+y\bbi)\bbj}  =  e^{e^{\theta\bbi}(x+y\bbi)\bbj e^{-\theta\bbi}}=e^{\theta\bbi}e^{(x+y\bbi)\bbj} e^{-\theta\bbi},$$ $e$ is equivariant with respect to the weight 2 rotation action of $S^1$ on $\RR^2$ ($z\cdot w=z^2w$ for $w\in \CC=\RR^2$).  

We promote $e$ to a map:
$$e:(\RR^2)^{2n-2}\to (S_\bbi^2)^{2n-2},~e(x_1,y_1,\dots, x_{2n-2},y_{2n-2})=(\bbi e^{x_1\bbj+y_1\bbk},\dots,\bbi e^{x_{2n-2}\bbj+y_{2n-2}\bbk}).
$$ 
Then $e$ is an equivariant diffeomorphism near $0$, taking $0$ to $(\bbi,\dots,\bbi)$. We show that the composite $g\circ e$ has a non-degenerate Hessian at $0$ with signature 0. Since 
$$\bbi e^{x_{2\ell+1}\bbj+y_{2\ell+1}\bbk} \bbi e^{x_{2\ell}\bbj+y_{2\ell}\bbk} =-e^{-(x_{2\ell+1}\bbj+y_{2\ell+1}\bbk)}e^{x_{2\ell}\bbj+y_{2\ell}\bbk}$$
it follows that 
\begin{equation}\label{cutout}\begin{multlined}\hskip.5in g\circ e(x_1,y_1,\dots, x_{2n-2},y_{2n-2})=
\Real(\bbi(
\bbi e^{x_1\bbj+y_1\bbk} \dots \bbi e^{x_{2n-2}\bbj+y_{2n-2}\bbk})\\
=(-1)^{n-1}\Real(\bbi \textstyle \prod\limits_{\ell=1}^{n-1} e^{-(x_{2\ell-1}\bbj+y_{2\ell-1}\bbk)}e^{x_{2\ell}\bbj+y_{2\ell}\bbk})\\
=(-1)^{n-1}\Real\big( \bbi \textstyle \prod\limits_{\ell=1}^{n-1}(1- x_{2\ell-1}\bbj-y_{2\ell-1}\bbk+\frac{(x_{2\ell-1}\bbj+y_{2\ell-1}\bbk)^2}{2})(1+x_{2\ell}\bbj+y_{2\ell}\bbk+\frac{(x_{2\ell}\bbj+y_{2\ell}\bbk)^2}{2})\big)+ O(3).
\end{multlined}\end{equation}
Since $\Real(\bbi\frac{(x_{\ell}\bbj+y_{\ell}\bbk)^2}{2}))=0=\Real(\bbi (x_{\ell}\bbj+y_{\ell}\bbk)  ))$, 
$$\begin{multlined} g\circ e(x_1,y_1,\dots, x_{2n-2},y_{2n-2})=
(-1)^{n-1}\Real\big(\bbi\textstyle\prod\limits_{\ell=1}^{n-1}(1- x_{2\ell-1}\bbj-y_{2\ell-1}\bbk)(1+x_{2\ell}\bbj+y_{2\ell}\bbk)\big)+ O(3)\\
=(-1)^{n-1}\Real\big(\bbi\textstyle\sum\limits_{\ell<m} (-1)^{\ell+m}(  x_{\ell}\bbj+y_{\ell}\bbk)( x_{m}\bbj+y_{m}\bbk)\big)+O(3)\\
= (-1)^{n-1}\textstyle\sum\limits_{\ell<m} (-1)^{\ell+m}(y_{\ell}x_m-x_\ell y_m) +O(3)\\
\end{multlined}$$

Reorder the variables in the order $(x_1,\dots,x_{2n-2}, y_1,\dots, y_{2n-2})$. Then the Hessian of $(-1)^{n-1}g\circ e$ takes the form 
$$\begin{pmatrix}0&A\\ A^{T}&0\end{pmatrix}$$
where $A$ is a $(2m-2)\times (2m-2)$ matrix with 
$$A_{i,j}= (-1)^{n-1}\frac{\partial^2(g\circ e)}{\partial y_j\partial x_i}=\begin{cases} 0& \text {if } i=j,\\
                         (-1)^{i+j+1}& \text {if } i>j,\\
                        (-1)^{i+j} & \text {if } i<j.\\ \end{cases}
$$
If $A$ is invertible, then the Hessian of $g\circ e$ is non-degenerate with signature 0. The mod 2 reduction $B$ of $A$ is a matrix with $0$ in the diagonal entries and $1$s in all off diagonal entries. Then $B^2$ is given by 
$$(B^2)_{i,j}=\sum_{\ell=1}^{2n-2} B_{i,\ell} B_{\ell j}= \begin{cases} 2n-4&\text{ if } i\ne j,\\ 2n-3& \text{ if } i= j\end{cases}$$
That is, $B^2$ is the identity matrix (over $\FF_2$), and so $\det(A^2)$ is an odd integer, hence non-zero, so that $A$ is invertible.

The Morse lemma now shows that after a change of coordinates $\phi:(\RR^2)^{2n-2}\cong (\RR^2)^{2n-2}$, 
$$g\circ e\circ \phi(u_1,\dots, u_{4n-4})=\sum_{i=1}^{2n-2}u_i^2- \sum_{i=2n-1}^{4n-4}u_i^2$$
in a neighborhood of $0$. Thus, near $(\bbi,\dots,\bbi)$, $g^{-1}(0)$   is a cone on $S^{2n-3}\times S^{2n-3}$. Since $S^1/\{\pm 1\}$ acts freely away from the cone point, a neighborhood of $(\bbi,\dots,\bbi)$ in $g^{-1}(0)/S^1\cong R(S^2, 2n)$ is isomorphic to  a cone on $(S^{2n-3}\times S^{2n-3})/S^1$ for some free action of $S^1$. As we described before, this shows that each point in $R(S^2,2n)_{\rm ab}$ has such a neighborhood in $R(S^2, 2n)$.
\end{proof}

\section{Proof of Theorem \ref{prop2.2}
}\label{apC}

\begin{proof}[Proof of Theorem \ref{prop2.2}] Conjugation by $\bbk$ leaves the binary dihedral subgroup $B$ invariant. Moreover it fixes the subgroup $\{e^{\theta\bbk}\}$ and acts by $-1$ on the coset $\{e^{\theta\bbk}\bbi\}$. hence if $\rho\in R(S^2,2n)$ takes values in $B$, then $\alpha(\gamma)\rho(\gamma)=\bbk\rho(\gamma)(-\bbk)$ and so  $\alpha_*$ fixes $R(S^2, 2n)_B$.

Conversely, choose  a conjugacy class $[\rho]\in R(S^2,2n)$ fixed by $\alpha_*$ and let $\rho$ be a representation in this class. Hence there exists a $g\in SU(2)$ so that $\rho(x_i)\alpha(x_i)=g\rho(x_i)g^{-1}$ for all $i$. Since $\alpha(x_i)=-1$, this is equivalent to $ [g,\rho(x_i)]=-1$.  It is easy to see that  $a,b\in SU(2)$ satisfy $[a,b]=-1$ if and only if $\Real(a)=\Real(b)=\Real(ab)=0$.  Thus $\Real(g)=0$ and $\Real(g\rho(x_i))=0$ for all $i$. 
 By conjugating, we may assume that $g=\bbk$. The equation $\Real(\bbk\rho(x_i))=0$ implies that each $\rho(x_i)$ lies in the circle $\{e^{\theta\bbk}\bbi\}$. Hence the generators all are sent to $B$ and so $[\rho]\in R(S^2,2n)_B$.   
 
 \medskip
For the second statement,  suppose that $[\rho_0],[\rho_1]\in R(S^2,2n)$ satisfy $\tp([\rho_0])=\tp([\rho_1])$. By Proposition \ref{prop2.1} this happens if and only if $p^*([\rho_0])=p^*([\rho_1])$, that is, if and only if the restrictions of $\rho_0$ and $\rho_1$ to $\ker \alpha$ are conjugate. By conjugating $\rho_1$ if necessary, we may assume that the representatives are chosen so that the restrictions $\rho_0|_{\ker\alpha}$ and $\rho_1|_{\ker\alpha}$  agree. 
For convenience denote $\pi_1(S^2\setminus \{a_i\})$ simply by $\pi$. 

Consider first the case when $\rho_0(x_1)=\pm \rho_1(x_1)$.  Since $x_ix_1^{-1}\in \ker\alpha$,  
\begin{equation}\label{eq5}
\rho_0(x_i)=\rho_0(x_ix_1^{-1})\rho_0(x_1)=
 \rho_1(x_ix_1^{-1}) \rho_0(x_1)=
 \begin{cases} \rho_1(x_i)&\text{ if } \rho_0(x_1)=\rho_1(x_1)\\
\alpha(x_i) \rho_1(x_i)&\text{ if } \rho_0(x_1)=-\rho_1(x_1)\end{cases}
\end{equation}
and hence either $\rho_0=\rho_1$ or $\rho_0=\alpha_*(\rho_1)$.

Consider next the case when $\rho_0(x_1)\ne \pm \rho_1(x_1)$. Since  $\Real(\rho_0(x_1))=0=\Real(\rho_1(x_1))$, there exists a $g\in SU(2)$ so that after conjugating both $\rho_0$ and $\rho_1$ by $g$, 
 $$\rho_0(x_1)=\bbi\text{ and } \rho_1(x_1)=e^{\theta\bbk}\bbi $$ for some $\theta\in (0,\pi)$.
 
 Then 
 \begin{equation}
\label{eq6} 0=\Real(\rho_0(x_i))=\Real(\rho_0(x_ix_1^{-1})\rho_0(x_1))=\Real(\rho_0(x_ix_1^{-1})\bbi)
\end{equation}
 and  similarly   
   \begin{equation}
\label{eq7}  0=\Real(\rho_1(x_ix_1^{-1})e^{\theta\bbk}\bbi). 
\end{equation}

Since $\alpha(x_ix_1^{-1})=1$, $\rho_0(x_ix_1^{-1})=\rho_1(x_ix_1^{-1})$ and so Equations (\ref{eq6}) and (\ref{eq7}) together with the fact that $e^{\theta\bbk}\ne \pm 1$ imply that $\rho_0(x_ix_1^{-1})=e^{\beta_i\bbk}$ for some $\beta_i$.  Then 
$$\rho_0(x_i)=\rho_0(x_ix_1^{-1})\rho_0(x_1)=e^{\beta_i\bbk}\bbi$$ and $$\rho_1(x_i)=\rho_1(x_ix_1^{-1})\rho_0(x_1)=e^{\beta_i\bbk}e^{\theta\bbk}\bbi=e^{\tfrac{1}{2}\theta\bbk}\rho_0(x_i)e^{-\tfrac{1}{2}\theta\bbk}.$$
Since the $x_i$ generate $\pi$, this shows that $\rho_0$ and $\rho_1$ are conjugate, i.e. $[\rho_0]=[\rho_1]$. Notice that in this case $[\rho_0],[\rho_1]\in R(S^2,2n)_B$.

\medskip

Finally, we prove the third statement. By definition $[\rho]\in R(S^2, 2n)_B$ is the conjugacy class of a representation $\rho:\pi\to B=\{e^{\theta\bbk}\}\cup\{e^{\theta\bbk}\bbi\}$ which takes each $x_i$ into the coset $\{e^{\theta\bbk}\bbi\}$. The composite of $\rho$ with the obvious surjective homomorphism $B\to \{\pm 1\}$ is precisely $\alpha$, hence $\rho(\ker\alpha)\subset \{e^{\theta\bbk}\}$, this shows that $p^*(\rho)$ is abelian. Clearly $c_*$ and $i^*$ take representations with abelian image  to representations with abelian image, showing that 
$\tp(R(S^2,2n)_B)\subset  R(F_{n-1})_{\rm ab}$.

Conversely, suppose $\tp(\rho)\in  R(F_{n-1})_{\rm ab}$. We may conjugate $\rho$ so that $\tp(\rho)\subset \{e^{\theta\bbk}\}$. 

 If $\rho(x_i)=\pm\rho(x_j)$ for each pair of indices $i,j$, then $\rho$ is abelian, hence lies in $ R(S^2,2n)_B$. 
Suppose instead for some indices $i,j$, $\rho(x_i)\ne\pm\rho(x_j)$. Then   $\rho(x_ix_j)=e^{\beta\bbk}$ for some $\beta$ satisfying $\sin\theta\ne 0$. Since
$$0=-\Real(\rho(x_j))=\Real(\rho(x_i)^2\rho(x_j))=\Real(\rho(x_i)\rho(x_ix_j))=
\Real(\rho(x_i)e^{\beta\bbk} )=\sin\beta\Real(\rho(x_i)\bbk),
$$
 it follows that $\rho(x_i)$ is perpendicular to $\bbk$, and hence lies in $\{e^{\theta\bbk}\bbi\}$. A similar computation shows that  $\rho(x_j)$ lies in $\{e^{\theta\bbk}\bbi\}$. If $\ell$ is another index, then either $\rho(x_\ell)=\pm \rho(x_i)$, and hence $\rho(x_\ell)$ lies in $\{e^{\theta\bbk}\bbi\}$, or $\rho(x_\ell)\ne\pm \rho(x_i)$ and, as before,    lies in $\{e^{\theta\bbk}\bbi\}$. We conclude that $\rho\in R(S^2,2n)_B$, as desired.
 \end{proof}


\end{document}

%% file: figures/genus2.pdf_tex
\begingroup%
  \makeatletter%
  \providecommand\color[2][]{%
    \errmessage{(Inkscape) Color is used for the text in Inkscape, but the package 'color.sty' is not loaded}%
    \renewcommand\color[2][]{}%
  }%
  \providecommand\transparent[1]{%
    \errmessage{(Inkscape) Transparency is used (non-zero) for the text in Inkscape, but the package 'transparent.sty' is not loaded}%
    \renewcommand\transparent[1]{}%
  }%
  \providecommand\rotatebox[2]{#2}%
  \ifx\svgwidth\undefined%
    \setlength{\unitlength}{516.13964844bp}%
    \ifx\svgscale\undefined%
      \relax%
    \else%
      \setlength{\unitlength}{\unitlength * \real{\svgscale}}%
    \fi%
  \else%
    \setlength{\unitlength}{\svgwidth}%
  \fi%
  \global\let\svgwidth\undefined%
  \global\let\svgscale\undefined%
  \makeatother%
  \begin{picture}(1,0.35624393)%
    \put(0,0){\includegraphics[width=\unitlength]{genus2.pdf}}%
    \put(0.01494624,0.212719){\color[rgb]{0,0,0}\makebox(0,0)[lb]{\smash{$y_1$}}}%
    \put(0.1141442,0.21891888){\color[rgb]{0,0,0}\makebox(0,0)[lb]{\smash{$y_2$}}}%
    \put(0.26382682,0.21714754){\color[rgb]{0,0,0}\makebox(0,0)[lb]{\smash{$y_3$}}}%
    \put(0.69604653,0.22157598){\color[rgb]{0,0,0}\makebox(0,0)[lb]{\smash{$y_4$}}}%
    \put(0.85015767,0.22423308){\color[rgb]{0,0,0}\makebox(0,0)[lb]{\smash{$y_5$}}}%
    \put(0.9502413,0.22954723){\color[rgb]{0,0,0}\makebox(0,0)[lb]{\smash{$y_6$}}}%
    \put(0.40996668,0.2454898){\color[rgb]{0,0,0}\makebox(0,0)[lb]{\smash{$s_1$}}}%
    \put(0.30811164,0.04089396){\color[rgb]{0,0,0}\makebox(0,0)[lb]{\smash{$r_1$}}}%
    \put(0.540164,0.24991828){\color[rgb]{0,0,0}\makebox(0,0)[lb]{\smash{$r_2$}}}%
    \put(0.65176166,0.03469409){\color[rgb]{0,0,0}\makebox(0,0)[lb]{\smash{$s_2$}}}%
  \end{picture}%
\endgroup%